\documentclass{article}
\usepackage{epsfig}
\usepackage{amsmath}
\usepackage{amsfonts}
\usepackage{amsthm}
\usepackage{algorithm}
\usepackage{algorithmic}
\usepackage{graphicx}
\usepackage{subfigure}

 \newtheorem{theorem}{Theorem}[section]
 \newtheorem{lemma}[theorem]{Lemma}

\title{ A Support Based Algorithm for Optimization with Eigenvalue Constraints }
\author{
Emre~Mengi\thanks{ Department of Mathematics, Ko\c{c} University,
Rumelifeneri Yolu, 34450 Sar{\i}yer-\.{I}stanbul, Turkey {\tt (emengi@ku.edu.tr)}. 
The work of the author was supported in part by the European Commision grant
PIRG-GA-268355.
 }
}

\begin{document}
\maketitle

\begin{abstract}
\noindent
Optimization of convex functions subject to eigenvalue constraints is intriguing because of 
peculiar analytical properties of eigenvalues, and is of practical interest because of wide
range of applications in fields such as structural design and control theory. Here we focus
on the optimization of a linear objective subject to a constraint on the smallest eigenvalue of
an analytical and Hermitian matrix-valued function. We offer a quadratic support function based
numerical solution. The quadratic support functions are derived utilizing the variational properties
of an eigenvalue over a set of Hermitian matrices. Then we establish the local convergence
of the algorithm under mild assumptions, and deduce a precise rate of convergence result
by viewing the algorithm as a fixed point iteration. We illustrate its applicability in practice on the 
pseudospectral functions.
 \\

  \\


\end{abstract}

\pagestyle{myheadings}
\thispagestyle{plain}
\markboth{ E. MENGI }{ OPTIMIZATION WITH EIGENVALUE CONSTRAINTS }

\section{Introduction}

Consider an analytic and Hermitian matrix-valued function ${\mathcal A}(\omega) : {\mathbb R}^d \rightarrow {\mathbb C}^{n\times n}$.
This work concerns optimization problems of the form
\begin{equation}\label{eq:problem}
	{\rm maximize} \;\;\; c^T \omega	\;\;\;\;\;\;\;\;	{\rm subject} \;\; {\rm to} \;\;\;\; \lambda_{\min} ({\mathcal A}(\omega)) \leq 0
\end{equation}
where $c \in {\mathbb R}^d$ is given, $\lambda_{\min}(\cdot)$ denotes the smallest eigenvalue, and where we assume
that the feasible set $\{ \omega \in {\mathbb R}^d \; | \; \lambda_{\min} ({\mathcal A}(\omega)) \leq 0 \}$ is bounded
in order to ensure the well-posedness of the problem. In a recent work \cite{Mengi2013}
we pursued support-function based ideas to optimize a prescribed eigenvalue of ${\mathcal A}(\omega)$ globally on a box
${\mathcal B} \subset {\mathbb R}^d$. Remarkably, such ideas yield a linearly convergent algorithm that overcomes non-convexity 
and non-smoothness intrinsic to involved eigenvalue functions. Here we explore the use of support functions for the numerical
solution of (\ref{eq:problem}).

The problem that we tackle is non-convex due to the eigenvalue constraint. We convexify the problem by replacing
the eigenvalue constraint with a convex quadratic constraint. Specifically, let $\omega_k$ be a feasible point satisfying
$\lambda_{\min}({\mathcal A}(\omega_k)) \leq 0$, we benefit from a quadratic support function about $\omega_k$ such
that 
\[
	q_k(\omega_k) =  \lambda_{\min}({\mathcal A}(\omega_k)) \;\;\; {\rm and}  \;\;\; q_k(\omega) \geq \lambda_{\min}({\mathcal A}(\omega))
\]
for all $\omega \in {\mathbb R}^d$. The original non-convex problem (\ref{eq:problem}) is replaced by
\begin{equation}\label{eq:convex_problem}
	{\rm maximize} \;\;\; c^T \omega	\;\;\;\;\;\;\;\;	{\rm subject} \;\; {\rm to} \;\;\;\; q_k(\omega) \leq 0.
\end{equation}
The convex problem above can be solved analytically, and has a maximizer $\omega_\ast$ that is suboptimal yet
feasible with respect to the original problem. We build a new quadratic support function $q_{k+1}(\omega)$ about the maximizer 
$\omega_{k+1} = \omega_\ast$, replace the constraint in (\ref{eq:convex_problem}) with $q_{k+1}(\omega) \leq 0$, and
solve the updated convex optimization problem. The practicality of the algorithm relies on a $\gamma$ satisfying
\begin{equation}\label{eq:sd_bound}
		\lambda_{\max} \left[ \nabla^2 \lambda_{\min} ({\mathcal A}(\omega))   \right]		\leq 		\gamma	\;\;	\forall \omega \in {\mathbb R}^d
		\;\; {\rm such}\; {\rm that} \;\;  \lambda_{\min} ({\mathcal A}(\omega)) \;\; {\rm is} \; {\rm simple}.
\end{equation}
Above and throughout the text $\lambda_{\max}(\cdot)$ refers to the largest eigenvalue.
As we shall see in Section \ref{sec:algorithm}, the quadratic support functions are built on the existence of such a $\gamma$. 
It seems feasible to deduce such $\gamma$ analytically; this is discussed in Section \ref{sec:sec_der}.

Optimization based on support functions dates back to Kelley's cutting plane method \cite{Kelley1960}, where linear 
objectives subject to convex constraints are solved by exploiting support functions. Kelley's idea is really
suitable for convex constraints, and extensions to non-convex constraints to locate a locally optimal solution
do not appear straightforward. More recently, inspired from Kelley's cutting plane method, bundle methods became popular 
for non-smooth optimization built around linear support functions defined in terms of subgradients \cite{Kiwiel1985, Makela2002}. 
Bundle methods are especially effective for the unconstrained optimization of a convex non-smooth function, yet
they are not as effective in the non-convex setting even to locate a locally optimal solution.

There are various applications that fit into the setting (\ref{eq:problem}). For instance, the $\epsilon$-pseudospectral
abscissa and radius \cite{Trefethen2005} of a given matrix $A$, which are the rightmost and outermost points in the 
$\epsilon$-pseudospectrum of the matrix (the set comprised of the eigenvalues of all matrices within an 
$\epsilon$-neighborhood with respect to the matrix 2-norm) fit perfectly into the setting; Section \ref{sec:pseudo} 
illustrates this. In structural design, \cite{Achtziger2007} a classical problem is the minimization of the volume 
subject to inequality constraints on the smallest eigenvalue. In robust control theory, it is desirable
to design a system subject to the largest eigenvalue not exceeding a prescribed tolerance.

The paper is organized as follows. In the next section we derive support functions for $\lambda_{\min}({\mathcal A}(\omega))$,
and specify the algorithm based on the solution of the convex problem (\ref{eq:convex_problem}). 
In Section \ref{sec:convergence} we show that the sequence generated by the algorithm converges to 
a locally maximal solution of the non-convex problem (\ref{eq:problem}) under mild assumptions.
Section \ref{sec:rate_convergence} is devoted to a rate-of-convergence analysis of the algorithm,
by viewing the algorithm as a fixed-point iteration. The practicality of the algorithm relies on the
deduction of an upper bound $\gamma$ satisfying (\ref{eq:sd_bound}) either analytically or numerically.
We present a result in Section \ref{sec:sec_der} that facilitates deducing such $\gamma$ analytically.
Finally, we illustrate the practical usefulness of the algorithm on the pseudospectral functions 
in Section \ref{sec:pseudo}, though the field of applicability of the algorithm is wider than the
pseudospectral functions.

% We elaborate on the framework (\ref{eq:problem}) due to its simplicity. In particular it seems feasible to deduce 
% global upper bounds $\gamma$ satisfying (\ref{eq:sd_bound}) analytically as discussed
% in Section \ref{sec:sec_der}. The latter two applications mentioned above would require the solution of
% problems of the form (\ref{eq:problem}), but with the direction of the inequality in the constraint reversed.
% In such cases, the algorithm and its convergence theory are applicable, but with the support functions bounding 
% from below.

\section{Derivation of the Algorithm}\label{sec:algorithm}

We begin with the derivation of the support functions. The derivation depends on the analytical 
properties of $\lambda_{\min} \left( {\mathcal A}(\omega) \right)$. 
We summarize the relevant classical results below \cite{Rellich1969, Lancaster1964}.

\begin{lemma}\label{lemma:anal_eig}
Let ${\mathcal A}(\omega) : {\mathbb R}^d 	\rightarrow 	{\mathbb C}^{n\times n}$ be Hermitian and analytic, and 
$\Phi : {\mathbb R} \rightarrow {\mathbb C}^{n\times n}$ be defined by $\Phi(\alpha) := {\mathcal A}(\hat{\omega} + \alpha p)$ 
for given $\hat{\omega}, p \in {\mathbb R}^d$. Then the following hold:
\begin{enumerate}
	\item[\bf (i)] There is an ordering $\phi_1(\alpha), \dots, \phi_n(\alpha)$ of the eigenvalues of 
	$\Phi(\alpha)$ so that each eigenvalue $\phi_j(\alpha)$ for $j = 1,\dots,n$ is analytic on ${\mathbb R}$;
	\item[\bf (ii)] Suppose that $\phi(\alpha) := \lambda_{\min}(\Phi(\alpha))$ 
	is simple for all $\alpha$ on an open interval ${\mathcal I}$ in ${\mathbb R}$. Then $\phi(\alpha)$ is analytic
	on ${\mathcal I}$;
	\item[\bf (iii)] The left-hand $\phi'_{-}(\alpha)$ and the right-hand $\phi'_{+}(\alpha)$ derivatives of $\phi(\alpha) := \lambda_{\min}(\Phi(\alpha))$
	exist everywhere. Furthermore $\phi'_{-}(\alpha) \geq \phi'_{+}(\alpha)$ at all $\alpha \in {\mathbb R}$;
	\item[\bf (iv)] The eigenvalue function $\lambda_{\min}({\mathcal A}(\omega))$ is twice continuously differentiable at
	all $\omega\in {\mathbb R}^d$ where it is simple. Furthermore, at all such $\omega$ we have
	\[
		\frac{\partial \lambda_{\min}({\mathcal A}(\omega))}{\partial \omega_j}
					=
			v_n(\omega)^\ast
				\frac{\partial {\mathcal A}(\omega)}{\partial \omega_j}
			v_n(\omega),
	\]
	and
	\begin{eqnarray*}
		\frac{\partial^2 \lambda_{\min}({\mathcal A}(\omega))}{\partial \omega_k \; \partial \omega_\ell}
						= &
		v_n^{\ast}(\omega)	\frac{\partial^2 {\mathcal A}(\omega)}{\partial \omega_k \; \partial \omega_l} v_n(\omega)
							+  \hskip 52ex \\
						&
		2\cdot \Re
			\left[
					\sum_{m=1}^{n-1}	
						\frac{1}{ \lambda_{\min}({\mathcal A}(\omega)) - \lambda_j(\omega) }  
						\left(
							v_n(\omega)^{\ast} \frac{\partial {\mathcal A}(\omega)}{\partial \omega_k} v_m(\omega)
						\right) 
						\left(
							v_m(\omega)^{\ast} \frac{\partial {\mathcal A}(\omega)}{\partial \omega_\ell} v_n(\omega)
						\right)
			\right].
	\end{eqnarray*}
	where $\lambda_j(\omega)$ denotes the $j$th largest eigenvalue of ${\mathcal A}(\omega)$ and $v_j(\omega)$ denotes
	an associated eigenvector that is analytic along every line in ${\mathbb R}^d$ such that $\{ v_1(\omega), \dots, v_n(\omega) \}$ is
	orthonormal.
\end{enumerate}
\end{lemma}
\noindent
In the theorem below and elsewhere $\| \cdot \|$ denotes the 2-norm or Euclidean norm on ${\mathbb R}^d$.
\begin{theorem}[Support Functions]
Suppose ${\mathcal A}(\omega) : {\mathbb R}^d 	\rightarrow 	{\mathbb C}^{n\times n}$ is Hermitian and analytic,
$\gamma \in {\mathbb R}$ satisfies (\ref{eq:sd_bound}), and $\omega_k \in {\mathbb R}^d$ is such that $\lambda_{\min}({\mathcal A}(\omega_k))$ 
is simple. Then
\[
    \lambda_{\min}({\mathcal A}(\omega))  \;\;  \leq  \;\; 
    	q_k(\omega) := \lambda_k + \nabla \lambda_k^T (\omega - \omega_k) + \frac{\gamma}{2} \| \omega - \omega_k \|^2  \;\;\;\; \forall \omega \in {\mathbb R}^d
\]
where $\lambda_k := \lambda_{\min}({\mathcal A}(\omega_k))$ and $\nabla \lambda_k := \nabla \lambda_{\min}({\mathcal A}(\omega_k))$.     
\end{theorem}
\begin{proof}
Let $p = (\omega - \omega_k)/\|\omega - \omega_k\|$, and define $\phi(\alpha) := \lambda_{\min} \left( {\mathcal A}(\omega_k + \alpha p) \right)$.
Denote the points on $(0,\|\omega - \omega_k\|)$ where $\lambda_{\min}\left( {\mathcal A}(\omega_k + \alpha p) \right)$ is not simple with 
$\alpha_1, \dots, \alpha_m$. There are finitely many such points, because $\lambda_{\min}\left( {\mathcal A}(\omega_k + \alpha p) \right)$ is 
the minimum of $n$ analytic functions from part \textbf{(i)} of Lemma \ref{lemma:anal_eig}. Indeed, two analytic functions are identical or can intersect 
each other at finitely many points on a finite interval.

Partition the open interval $(0,\|\omega - \omega_k\|)$ into open subintervals ${\mathcal I}_j := (\alpha_{j},\alpha_{j+1})$ for $j = 0,\dots, m$ where
$\alpha_0 = 0, \alpha_{m+1} = \|\omega - \omega_k\|$. On each open subinterval ${\mathcal I}_j := (\alpha_{j-1},\alpha_j)$ the function $\phi(\alpha)$ is 
analytic from part \textbf{(ii)} of Lemma \ref{lemma:anal_eig}. Indeed $\phi(\alpha) = \phi_{i_j}(\alpha)$ on the closure of ${\mathcal I}_j$ for some
analytic $\phi_{i_j}(\alpha)$ stated in part \textbf{(i)} of Lemma \ref{lemma:anal_eig}, moreover $\phi'_+(\alpha_j) = \phi'_{i_j}(\alpha_j)$.
Thus applying the Taylor's theorem for each $\alpha$ on the closure of ${\mathcal I}_j$ we deduce
\begin{equation}\label{eq:quad_ineq0}
\begin{split}
	\phi(\alpha)	= &	\phi(\alpha_j)	+	\phi'_+(\alpha_j) (\alpha - \alpha_j)	+	\frac{\phi''(\eta)}{2} (\alpha - \alpha_j)^2  \\
			= &	\phi(\alpha_j)	+	\phi'_+(\alpha_j) (\alpha - \alpha_j)	+	\frac{p^T \nabla^2 \lambda_{\min} ({\mathcal A}(\omega_k + \eta p)) p }{2} (\alpha - \alpha_j)^2 \\
				\leq &  \phi(\alpha_j)	+	\phi'_+(\alpha_j) (\alpha - \alpha_j)	+	\frac{ \gamma }{2} (\alpha - \alpha_j)^2 
\end{split}
\end{equation}
for some $\eta \in {\mathcal I}_j$, where the second equality is due to the twice differentiability of $\lambda_{\min} ({\mathcal A}(\omega))$ 
(part \textbf{(iv)} of Lemma \ref{lemma:anal_eig}), and the last inequality is due to 
$p^T \nabla^2 \lambda_{\min} ({\mathcal A}(\omega_k + \eta p)) p \leq \lambda_{\max} \left[ \nabla^2 \lambda_{\min} ({\mathcal A}(\omega_k + \eta p)) \right]$ 
and (\ref{eq:sd_bound}). Differentiating both sides of inequality (\ref{eq:quad_ineq0}) also yields
\begin{equation}\label{eq:quad_ineq00}
	\phi'(\alpha)	\leq 		\phi'_+(\alpha_j)	+	\gamma (\alpha - \alpha_j).
\end{equation}

Next we claim
% , noting that $\phi(\alpha_{m+1}) = \lambda_{\min}({\mathcal A}(\omega))$, we claim
\begin{equation}\label{eq:quad_ineq}
	\phi(\alpha_{m+1})	\leq	\phi(\alpha_{j})	+	\phi'_+(\alpha_{j}) (\alpha _{m+1}  - \alpha_{j})		+	\frac{ \gamma }{2} (\alpha_{m+1} - \alpha_{j})^2
\end{equation}
for $j = 0,\dots,m$. This is certainly true for $j = m$ from (\ref{eq:quad_ineq0}) on the interval ${\mathcal I}_{m}$ and with 
$\alpha = \alpha_{m+1}$. Suppose that inequality (\ref{eq:quad_ineq}) holds for $k \geq 1$. Exploiting 
$\phi'_+(\alpha_{k}) \leq   \phi'_{-}(\alpha_{k})$ (see part \textbf{(iii)} of Lemma \ref{lemma:anal_eig}), and then applying 
inequalities (\ref{eq:quad_ineq0}) and (\ref{eq:quad_ineq00}) with $\alpha = \alpha_k$ on the interval ${\mathcal I}_{k-1}$
leads us to 
\begin{eqnarray*}
\phi(\alpha_{m+1})	\leq & \phi(\alpha_{k})	+	\phi'_{-}(\alpha_{k}) (\alpha _{m+1}  - \alpha_{k})		+	\frac{ \gamma }{2} (\alpha_{m+1} - \alpha_{k})^2 \hskip 13ex \\
				\leq & \phi(\alpha_{k-1})	+	\phi'_+(\alpha_{k-1}) (\alpha _{k}  - \alpha_{k-1})		+	\frac{ \gamma }{2} (\alpha_{k} - \alpha_{k-1})^2  \hskip 9ex \\
		& +  [\phi'_{+}(\alpha_{k-1}) + \gamma (\alpha_{k} - \alpha_{k-1})] (\alpha _{m+1}  - \alpha_{k})	+	\frac{ \gamma }{2} (\alpha_{m+1} - \alpha_{k})^2  \\
		= & \phi(\alpha_{k-1})	+	\phi'_+(\alpha_{k-1}) (\alpha _{m+1}  - \alpha_{k-1})		+	\frac{ \gamma }{2} (\alpha_{m+1} - \alpha_{k-1})^2. \hskip 4ex
\end{eqnarray*}
Thus by induction we conclude
\[
	\phi(\alpha_{m+1})	\leq	\phi(\alpha_{0})	   +	\phi'_+(\alpha_{0}) (\alpha _{m+1}  - \alpha_{0})		+	\frac{ \gamma }{2} (\alpha_{m+1} - \alpha_{0})^2.
\]
Recalling $\phi(\alpha_0) = \phi(0) = \lambda_k$, $\phi'_+(\alpha_{0}) = \phi'(0) = \nabla \lambda_k^T p$ and 
$(\alpha _{m+1}  - \alpha_{0}) = \| \omega - \omega_k \|$, we obtain the desired inequality
\[
	\lambda_{\min}({\mathcal A}(\omega)) \leq \lambda_k  +   \nabla \lambda_k^T (\omega - \omega_k)   + \frac{\gamma}{2} \| \omega - \omega_k \|^2.
\]
\end{proof}

Given a feasible point $\omega_0 \in {\mathbb R}^d$ satisfying $\lambda_{\min} ({\mathcal A}(\omega_0)) \leq 0$, 
the algorithm generates a sequence $\{ \omega_k \}$ of feasible points in ${\mathbb R}^d$. The update of $\omega_k$ 
is based on the solution of the following convex optimization problem:
\begin{equation}\label{eq:precise_convex_problem}
{\rm maximize} \;\;\; c^T \omega	\;\;\;\;\;\;\;\;	{\rm subject} \;\; {\rm to} \;\;\;\; q_{k}(\omega) := \lambda_k + \nabla \lambda_k^T (\omega - \omega_k) + \frac{\gamma}{2} \| \omega - \omega_k \|^2 \leq 0.
\end{equation}
The next point $\omega_{k+1}$ is defined to be the unique maximizer of the problem above. Notice that
the feasible set of the convex problem (\ref{eq:precise_convex_problem}) is contained inside the
feasible set of the original problem (\ref{eq:problem}), i.e.,
\[
	{\mathcal F}_k	:=	\{	\omega \in {\mathbb R}^d	\; | \;	q_k(\omega) \leq 0		\}
			\;\;\;\;\;			\subseteq			\;\;\;\;\;
	{\mathcal F}	:=	\{	\omega \in {\mathbb R}^d	\; | \;	\lambda(\omega) \leq 0		\}.
\]
Thus $\omega_{k+1} \in {\mathcal F}_k \subseteq {\mathcal F}$ remains feasible. It is assumed by the
algorithm that $\lambda_{\min}({\mathcal A}(\omega_k))$ is simple for each $\omega_k$ in the sequence, 
which introduces no difficulties in practice, since a point $\omega$ such that $\lambda_{\min}({\mathcal A}(\omega)$ 
is not simple is isolated in ${\mathbb R}^d$. On the other hand, being close to a multiple eigenvalue
does not cause any harm.

There is a degenerate case for the convergence of the algorithm as suggested by the following observation.
\begin{theorem}
Suppose the maximizer $\omega_\ast$ of (\ref{eq:precise_convex_problem}) is such that $\nabla q_{k}(\omega_\ast) = 0$.
Then \textbf{(i)} $\omega_{\ast} = \omega_k$ and \textbf{(ii)} $\nabla \lambda_k = 0$.
\end{theorem}
\begin{proof}
The maximizer $\omega_\ast$ of (\ref{eq:precise_convex_problem}) must be attained on the boundary of ${\mathcal F}_k$,
that is $q_k(\omega_\ast) = 0$. Assuming 
\begin{equation}\label{eq:zero_quad_grad}
	\nabla q_k(\omega_\ast) = \nabla \lambda_k + \gamma (\omega_\ast - \omega_k) = 0
\end{equation}
yields
$
		q_k(\omega_\ast) = \lambda_k - \frac{\gamma}{2} \| \omega_\ast - \omega_k \|^2 = 0.
$ 
Furthermore, since $\omega_k$ is feasible, $\lambda_k = \lambda_{\min}({\mathcal A}(\omega_k)) \leq 0$.
This would imply $\lambda_k = 0$ and $\omega_\ast = \omega_k$. Now the second assertion follows 
from (\ref{eq:zero_quad_grad}).
\end{proof}
\noindent
The first assertion of the theorem above means that $\omega_{s} = \omega_k$ and $q_{s}(\omega) \equiv q_{k}(\omega)$
for each $s > k$. Thus convergence to a point $\omega_\ast$ such that $\nabla \lambda({\mathcal A}(\omega_\ast)) = 0$
seems possible. We rule this out by assuming $\nabla \lambda_k \neq 0$ for each $k$.

Now we apply the Karush-Kuhn-Tucker conditions to the constrained problem (\ref{eq:precise_convex_problem}).
The maximizer $\omega_\ast$ must satisfy
\begin{equation}\label{eq:KKT}
	c	=	\mu \nabla q_k(\omega_\ast)	\;\;\; 	{\rm and}	\;\;\;		q_k(\omega_\ast) = 0
\end{equation}
for some positive $\mu \in {\mathbb R}$. Solving the equations above for $\omega_\ast$ and the positive scalar $\mu$,
and setting $\omega_{k+1} = \omega_\ast$, leads us to the update rule
\begin{equation}\label{eq:update_rule}
	\omega_{k+1} 		=	\omega_k		+	\frac{1}{\gamma}	\left[		\frac{1}{\mu_+} \cdot c 	-	\nabla \lambda_k	\right],
	\;\;\; {\rm where}		\;\;\;
	\mu_+
			=
	\frac{ \| c \| }{ \sqrt{   \| \nabla \lambda_k \|^2	-	2\gamma \lambda_k   } }.	
\end{equation}
% where
% \begin{equation}\label{eq:Lagrange_mult}
%	\mu_+
%			=
%	\sqrt{\frac{ \| c \|^2 }{ \| \nabla \lambda_k \|^2	-	2\gamma \lambda_k}}.
% \end{equation}
%The first equation above yields
%\[
%	(\omega_\ast - \omega_k)		=	\frac{1}{\gamma}	\left[ 		\frac{1}{\mu} \cdot c 	-	\nabla \lambda_k	\right].
%\]
%Plugging the right-hand side for $(\omega_\ast - \omega_k)$ in (\ref{eq:KKT}) leads to a quadratic equation in $\mu$,
%whose positive solution is given by
%\begin{equation}\label{eq:Lagrange_mult}
%	\mu_+
%			=
%	\sqrt{\frac{ \| c \|^2 }{ \| \nabla \lambda_k \|^2	-	2\gamma \lambda_k}}.
%\end{equation}
%Thus, setting $\omega_{k+1} = \omega_\ast$, we obtain the update rule
%\begin{equation}\label{eq:update_rule}
%	\omega_{k+1} 		=	\omega_k		+	\frac{1}{\gamma}	\left[		\frac{1}{\mu_+} \cdot c 	-	\nabla \lambda_k	\right]	
%\end{equation}
%where $\mu_+$ is given by (\ref{eq:Lagrange_mult}).

\section{Convergence}\label{sec:convergence}
We establish that the sequence $\{ \omega_k \}$ converges to a local maximizer 
of (\ref{eq:problem}) under mild assumptions. The proof depends on the existence of 
the so called feasible ascent directions from every feasible non-degenerate point that is not a local maximizer.
\begin{lemma}\label{lemma:ascent_directions}
Suppose $\omega_\ast \in {\mathcal F}$ is not a local maximizer of (\ref{eq:problem}), and is such that 
$\lambda_{\min}({\mathcal A}(\omega_\ast))$ is simple and $\nabla \lambda_{\min}({\mathcal A}(\omega_\ast)) \neq 0$.
Then there exists $p \in {\mathbb R}^d$ such that 
\begin{center}
	$c^T p > 0$ $\;\;$ and $\;\;$ $\nabla \lambda_{\min}({\mathcal A}(\omega_\ast))^Tp < 0$.
\end{center}
\end{lemma}  

The next theorem relates the local maximizer of problems (\ref{eq:problem}) and (\ref{eq:precise_convex_problem}).
\begin{theorem}
Suppose $\omega_k \in {\mathbb R}^d$ is such that $\lambda_{\min}({\mathcal A}(\omega_k))$ is simple and 
$\nabla \lambda_{\min}({\mathcal A}(\omega_k)) \neq 0$. The point $\omega_k$ is a local maximizer of (\ref{eq:problem}) 
if and only if it is a local maximizer of (\ref{eq:precise_convex_problem}).
\end{theorem}
\begin{proof}
If $\omega_k \in {\mathcal F}$ is a local maximizer of (\ref{eq:problem}), then there exists a $\delta > 0$ such that
\[
	c^T \omega_k \geq c^T \omega		\;\;\;\;\;\;\; \forall \omega \in {\mathcal B}(\omega_k,\delta) \cap {\mathcal F},
\]
where ${\mathcal B}(\omega_k,\delta) := \{ \omega \in {\mathbb R}^d \; | \: \| \omega - \omega_k \| \leq \delta \}$.
But notice that $\omega_k \in {\mathcal F}_k$ (i.e., $q_k(\omega_k) = 0$), and due to the property 
${\mathcal F}_k \subseteq {\mathcal F}$ we have
$
	c^T \omega_k \geq c^T \omega		\;\;\; \forall \omega \in {\mathcal B}(\omega_k,\delta) \cap {\mathcal F}_k,
$
meaning $\omega_k$ is a local maximizer of (\ref{eq:precise_convex_problem}).

On the other hand, if  $\omega_k \in {\mathcal F}$ is not a local maximizer of (\ref{eq:problem}), then there 
exists a direction $p \in {\mathbb R}^d$ such that $c^T p > 0$ and $\nabla \lambda_k^T p < 0$ due to 
Lemma \ref{lemma:ascent_directions}. But then, for all small $\alpha > 0$, we have
\[
	q_k(\omega_k + \alpha p) = \lambda_k + \nabla \lambda_k^T (\alpha p) + O(\alpha^2) < \lambda_k \leq 0,
		\;\;\;\;		{\rm and}		\;\;\;\;
		c^T(\omega_k + \alpha p) 	>	c^T \omega_k.
\]
Thus $\omega_k$ is not a local maximizer of (\ref{eq:precise_convex_problem}).
\end{proof}
\noindent
An implication of the theorem above is that if $\omega_k \in {\mathcal F}$ is a local maximizer of (\ref{eq:problem}), 
then $\omega_s = \omega_k$ for each $s > k$. Thus convergence to a local maximizer occurs. Unfortunately,
this type of finite convergence cannot happen if $\gamma$ defined by (\ref{eq:sd_bound}) is a strict upper bound,
since $\omega_k$ lies strictly inside ${\mathcal F}$ and thus cannot be a local maximizer of ($\ref{eq:problem})$.

For the main result in this section, we first observe that the values of the objective function must converge.
\begin{lemma}\label{thm:convergence_objective}
The sequence $\{ c^T \omega_k \}$ is convergent.
\end{lemma}
\begin{proof}
Notice that
\[
	\omega_{k+1}
				=
	\arg \max_{q_k(\omega) \leq 0}		c^T \omega
\]
Since $\omega_k$ is feasible with respect to the problem above, we must have $c^T \omega_{k+1} \geq c^T \omega_k$.
Thus the sequence $\{ c^T \omega_k \}$ is monotonically increasing. Moreover, denoting a global maximizer for
the original problem (\ref{eq:problem}) with $\omega_\ast$, due to ${\mathcal F}_k \subseteq {\mathcal F}$, 
we have $c^T \omega_\ast \geq c^T \omega_k$. Thus the sequence $\{ c^T \omega_k \}$ is also bounded
above, meaning the sequence must converge.
\end{proof}
It appears difficult at first to prove the convergence of $\{ \omega_k \}$, but, since each 
$\omega_k$ belongs to the bounded set ${\mathcal F}$, it must have convergent subsequences
by the Bolzano-Weirstrass theorem. First we establish the convergence of each of these subsequences to a 
local maximizer of (\ref{eq:problem}).
\begin{lemma}\label{lemma:conv_subsequences}
Consider a convergent subsequence of $\{ \omega_k \}$ with limit $\omega_\ast$ such that
$\lambda_{\min}({\mathcal A}(\omega_k))$, $\lambda_{\min}({\mathcal A}(\omega_\ast))$
are simple, and $\nabla \lambda_{\min}({\mathcal A}(\omega_k)) \neq 0$, $\nabla \lambda_{\min}({\mathcal A}(\omega_\ast)) \neq 0$
for each $k$. Then $\omega_\ast$ is a local maximizer of (\ref{eq:problem}).
\end{lemma}
\begin{proof}
% (*** FOR NOW THE PROOF ASSUMING $\lambda_{\min}({\mathcal A}(\omega_\ast))$ IS SIMPLE ***)
Let us denote the convergent subsequence of $\{ \omega_k \}$ with $\{ \omega_{k_j} \}$.
Let us also use the notations $\lambda_{k_j} = \lambda_{\min}({\mathcal A}(\omega_{k_j}))$,
$\lambda_\ast = \lambda_{\min}({\mathcal A}(\omega_\ast))$, and 
$\nabla \lambda_{k_j} = \nabla \lambda_{\min}({\mathcal A}(\omega_{k_j}))$,
$\nabla \lambda_\ast = \nabla \lambda_{\min}({\mathcal A}(\omega_\ast))$. Note that 
$\lim_{j\rightarrow \infty} \omega_{k_j} = \omega_\ast$ must be feasible, since all $\omega_k$ are feasible 
and $\lambda_{\min}({\mathcal A}(\omega))$ varies continuously with respect to $\omega$. 

Suppose for the sake of contradiction $\omega_{\ast}$ is not a local maximizer, then we infer from 
Lemma \ref{lemma:ascent_directions} the existence of a direction $p \in {\mathbb R}^d$ such that 
$c^T p = \eta > 0$ and $\nabla \lambda_\ast^T p = -\beta < 0$. Without loss of generality, we 
assume $\| p \| = 1$. There exists a ball ${\mathcal B}(\omega_\ast,\delta)$ 
such that $\nabla \lambda_{\min}({\mathcal A}(\omega))$ is simple for all 
$\omega \in {\mathcal B}(\omega_\ast,\delta)$. Furthermore, there exists an integer $j'$ such that 
each $\omega_{k_j}$ for $j \geq j'$ lies in ${\mathcal B}(\omega_\ast,\delta)$.
Due to part \textbf{(iv)} of Lemma \ref{lemma:anal_eig} (indicating the continuity of the
partial derivatives of $\lambda_{\min}({\mathcal A}(\omega))$ on ${\mathcal B}(\omega_\ast,\delta)$)
there also exists an integer $j'' \geq j'$ such that 
\begin{equation}\label{eq:neighborhood1}
	\| \nabla \lambda_\ast  -  \nabla \lambda_{k_j} \| \leq \beta/2	\;\;\;\;\;\;\;  \forall j \geq j''.
\end{equation}
 
Next we benefit from the convergence of $\{c^T \omega_k \}$ (Lemma \ref{thm:convergence_objective}).
In this respect we note that $\lim_{k \rightarrow \infty} c^T \omega_k = c^T \omega_\ast$.
For some $k'$ we must have 
\begin{equation}\label{eq:neighborhood2}
	0 \leq (c^T \omega_\ast - c^T \omega_k)  \leq  (\eta \beta)/(2 \gamma)	\;\;\;\;\;\;  \forall k \geq k'.
\end{equation} 
Now consider any $\omega_{k_j}$ such that $j \geq j''$ and $k_j \geq k'$. Recalling $\lambda_{k_j} \leq 0$, 
i.e., $\omega_{k_j}$ is feasible, the corresponding support function satisfies
\[
	q_{k_j}(\omega_{k_j} + \alpha p) = \lambda_{k_j} + \nabla \lambda_{k_j}^T (\alpha p) + \frac{\gamma}{2} \alpha^2 \leq 0
\]
for all $\alpha \in [0, (2/\gamma)(-\nabla \lambda_{k_j}^T p)]$. Furthermore,
\[
	\| \nabla \lambda_\ast 	-	\nabla \lambda_{k_j} \|  \geq  (\nabla \lambda_{k_j} - \nabla \lambda_\ast)^T p
			\;\;\;\;\;		\Longrightarrow		\;\;\;\;\;
	-\nabla \lambda_{k_j}^T p		\geq		-\nabla \lambda_{\ast}^T p		-	\| \nabla \lambda_\ast 	-	\nabla \lambda_{k_j} \|   \geq  \beta/2
\]
where the last inequality is due to (\ref{eq:neighborhood1}). Thus for $\tilde{\alpha} = (2/\gamma)(-\nabla \lambda_{k_j}^T p)$,
and employing (\ref{eq:neighborhood2}), we deduce
\begin{eqnarray*}
		c^T(\omega_{k_j}  +   \tilde{\alpha} p)	=	&	c^T \omega_{k_j}	+	(\tilde{\alpha}) c^T p	\\
										\geq	& 	c^T \omega_\ast + (\eta\beta)/(2\gamma).
\end{eqnarray*}
Since $\omega_{k_j}  +   \tilde{\alpha} p$ is a feasible point of (\ref{eq:precise_convex_problem}),
this implies that $\omega_{(k_j + 1)}$ is such that $c^T(\omega_{(k_j + 1)}) \geq c^T \omega_\ast + (\eta\beta)/(2\gamma)$
contradicting (\ref{eq:neighborhood2}). Thus, $\omega_\ast$ must be a local maximizer.
\end{proof}

Next we present the main convergence result, where a point $\omega_\ast \in {\mathbb R}^d$ is defined to be a
strict local maximizer if there exists a $\delta > 0$ such that $c^T \omega_{\ast} > c^T \omega$ for all 
$\omega \in [ {\mathcal B}(\omega_\ast,\delta) \cap {\mathcal F} ] \backslash \{ \omega_\ast \}$.
\begin{theorem}
Suppose that $\lambda_{\min}({\mathcal A}(\omega_k))$ are simple, and $\nabla \lambda_{\min}({\mathcal A}(\omega_k)) \neq 0$
for each $k$. Furthermore, suppose that each local maximizer $\omega_\ast$ of (\ref{eq:problem}) is strict, and such that
$\lambda_{\min}({\mathcal A}(\omega_\ast))$ is simple and $\nabla \lambda_{\min}({\mathcal A}(\omega_\ast)) \neq 0$.
Then the sequence $\{ \omega_k \}$ converges to a local maximizer of (\ref{eq:problem}).
\end{theorem}
\begin{proof}
Let $\{ \omega_{k_j} \}$ be a subsequence as in Lemma \ref{lemma:conv_subsequences} converging to a local
maximizer $\omega_\ast$, which is strict by assumption. Let ${\mathcal L}^\ast_{k}$ denote the connected component of 
the set
\[
		{\mathcal L}_k := \{ \omega \; | \; c^T \omega \geq  c^T \omega_{k} \} \cap {\mathcal F}
\]
containing $\omega_\ast$. Due to the continuity of $\lambda_{\min}({\mathcal A}(\omega))$, and since $\omega_\ast$
is a strict local maximizer, there exists an $\eta$ such that $c^T \omega_\ast > c^T \omega$ for all 
$\omega \in {\mathcal L}^\ast_{k_\eta} \backslash \{ \omega_\ast \}$.

We claim that all $\omega_k$ for $k \geq k_\eta$ belong to the set ${\mathcal L}^\ast_{k_\eta}$. To see this, first recall
that $\{ c^T \omega_k \}$ is monotonically increasing, and each $\omega_k$ is feasible. Thus $\omega_k \in {\mathcal L}_{k_\eta}$
for each $k \geq k_\eta$. Furthermore, $\omega = \omega_k + \alpha (\omega_{k+1} - \omega_k) \in {\mathcal L}_{k_\eta}$ 
for all $\alpha \in [0,1]$ and $k \geq k_\eta$. This is due to $c^T (\omega_{k+1} - \omega_k) \geq 0$ so that 
$c^T \omega \geq c^T \omega_k$, as well as $\omega_k, \omega_{k+1} \in {\mathcal F}_k$ and the convexity of ${\mathcal F}_k$
implying $\omega = \omega_k + \alpha (\omega_{k+1} - \omega_k) \in {\mathcal F}_k \subseteq {\mathcal F}$. Thus
each $\omega_k$ for $k \geq k_\eta$ belongs to the same connected component of ${\mathcal L}_{k_\eta}$.
Finally notice that $\omega_{k_\eta}$, thus all $\omega_k$ for $k \geq k_\eta$, must belong to the 
connected component of ${\mathcal L}_{k_\eta}$ containing $\omega_\ast$. Assuming otherwise yields
that $\| \omega_{k_j} - \omega_\ast \|$ for each $j \geq \eta$ is greater than or equal to the minimum distance 
from the component of ${\mathcal L}_{k_\eta}$ containing all $\omega_{k}$ for $k \geq k_\eta$ to $\omega_\ast$, 
which contradicts $\{ \omega_{k_j} \}$ converging to $\omega_\ast$.

For the sake of contradiction, assume that the sequence $\{ \omega_k \}$ does not converge to $\omega_\ast$.
Then there exists an $\epsilon > 0$ such that
\[
	\forall N \; \exists k \geq N	\;\;\;\;\;	\| \omega_k - \omega_\ast \|	\geq \epsilon.
\]
We can choose this $\epsilon$ as small as we like. On the other hand, since $\{ c^T \omega_k \}$ is convergent
by Lemma \ref{thm:convergence_objective} and indeed it must converge to $c^T \omega_\ast$, for any given positive 
integer $\ell$ there must exist a positive integer $N(\epsilon,\ell)$ such that
\[
	\forall k \geq N(\epsilon,\ell)	\;\;\;\;\;	c^T \omega_k \geq c^T \omega_\ast - \epsilon/\ell.
\]
Thus for each positive integer $\ell$ there exists a $k(\ell) \geq k_\eta$ such that
\begin{equation}\label{eq:nearby_pts_scomp}
	c^T \omega_{k(\ell)} \geq c^T \omega_\ast - \epsilon/\ell	\;\;\;	{\rm and}	\;\;\;	\| \omega_{k(\ell)} - \omega_\ast \|	\geq \epsilon.
\end{equation}
Now consider the maximal value of $c^T\omega$ over all $\omega \in \{ \tilde{\omega} \; | \; \| \tilde{\omega} - \omega_\ast \| \geq \epsilon \} \cap {\mathcal L}^\ast_{k_\eta}$.
This maximal value must be strictly less than $c^T \omega_\ast$, say $c^T \omega_\ast - \beta$ for some $\beta > 0$,
as the maximization is over a compact subset of ${\mathcal L}^\ast_{k_\eta}$. But all $\omega_{k(\ell)}$ belongs to
$\{ \tilde{\omega} \; | \; \| \tilde{\omega} - \omega_\ast \| \geq \epsilon \} \cap {\mathcal L}^\ast_{k_\eta}$.
In particular for any integer $\ell > \epsilon/\beta$ we end up with the contradiction 
$\omega_{k(\ell)} \in \{ \tilde{\omega} \; | \; \| \tilde{\omega} - \omega_\ast \| \geq \epsilon \} \cap {\mathcal L}^\ast_{k_\eta}$
and $c^T \omega_{k(\ell)} > c^T \omega_\ast - \beta$ from the inequalities in (\ref{eq:nearby_pts_scomp}).

\end{proof}

\section{A Fixed-Point View: Rate of Convergence }\label{sec:rate_convergence}
In this section, under the assumption that $\{ \omega_k \}$ itself converges to a local maximizer $\omega_\ast$, we deduce a linear rate of convergence
revealing also the factors affecting the speed of convergence. Throughout the section we use the short-hands $\lambda_\ast, \nabla \lambda_\ast$ 
and $\nabla^2 \lambda_\ast$ for $\lambda_{\min}( {\mathcal A}(\omega_\ast) ), \nabla \lambda_{\min}( {\mathcal A}(\omega_\ast) )$
and $\nabla^2 \lambda_{\min}( {\mathcal A}(\omega_\ast) )$, respectively. As we shall see, the rate of convergence 
is mainly determined by the eigenvalue distribution of the projected Hessian 
\[
	{\mathcal H}_V		:=	V^T \cdot \nabla^2 \lambda_\ast \cdot V 
\]
where $V \in {\mathbb R}^{d\times d-1}$ is an isometry with columns formed by an orthonormal basis for the 
subspace orthogonal to $\nabla \lambda_\ast$. In particular, the convergence 
is faster when the eigenvalues of ${\mathcal H}_V$ are closer to $\gamma$. In the extreme case, when 
${\mathcal H}_V = \gamma I$, the rate of convergence becomes superlinear.

We will put a fixed point theory in use: it follows from (\ref{eq:update_rule}) that the sequence $\{ \omega_k \}$
is a fixed point sequence $\omega_{k+1} = f(\omega_k)$ where
\begin{equation}\label{eq:fixed_point_func}
	f(\omega)
			:=
	\omega
			+
	\frac{1}{\gamma}
	\left[
		\frac{\sqrt{ \| \nabla \lambda_{\min} ( {\mathcal A}(\omega) ) \|^2   -   2\gamma \lambda_{\min}( {\mathcal A}(\omega) )}}{ \| c \|} \cdot c
						-
			\nabla \lambda_{\min}({\mathcal A}(\omega))	
	\right].
\end{equation}
The Jacobian of $f(\omega)$ given by
\[
	J(\omega)
		=
	I
			+
	\frac{1}{\gamma}
	\left[
			\frac{ c \cdot \nabla \lambda_{\min}({\mathcal A}(\omega))^T (\nabla^2 \lambda_{\min}({\mathcal A}(\omega))  -  \gamma I)}
			        { \| c \| \sqrt{ \| \nabla \lambda_{\min}({\mathcal A}(\omega) \| - 2\gamma \lambda_{\min}({\mathcal A}(\omega))}}
			        					-
			\nabla^2 \lambda_{\min}({\mathcal A}(\omega)) 
	\right]
\]
for $\omega$ close to $\omega_\ast$ plays a prominent role in determining the rate of convergence. The following second order 
necessary conditions must hold at the local maximizer $\omega_\ast$.
\begin{theorem}\label{thm:KKT}
Suppose that $\omega_\ast$ is a local maximizer of  (\ref{eq:optimal_Jacobian}) such that $\lambda_\ast$
is simple, and $\nabla \lambda_\ast \neq 0$. Then
	\textbf{(1)} $\lambda_\ast = 0$, 
	\textbf{(2)} $c = \mu \nabla \lambda_\ast$ for some $\mu > 0$, and
	\textbf{(3)} $V^T \nabla^2 \lambda_\ast V \succeq 0$ where $V \in {\mathbb C}^{d\times(d-1)}$ is a matrix 
	whose columns form an orthonormal basis for the subspace orthogonal to $\nabla \lambda_\ast$.
\end{theorem}
\noindent
It is straightforward to verify that any fixed point of $f(\omega)$ satisfies the first order optimality conditions (parts \textbf{(1)} and \textbf{(2)}
of Theorem \ref{thm:KKT}), or otherwise $\nabla \lambda_{\min}( {\mathcal A}(\omega) )$ vanishes at the fixed point.

At the local maximizer $\omega_\ast$, due to parts \textbf{(1)} and \textbf{(2)} of Theorem \ref{thm:KKT}, we must have 
$\lambda_\ast = 0$ and $c/\| c \|  = \nabla \lambda_\ast / \| \nabla \lambda_\ast \|$.
Thus, simple calculations yield
\begin{equation}\label{eq:optimal_Jacobian}
	J(\omega_\ast)
		= 
		\left[
			I 	-	\frac{ \nabla \lambda_\ast \cdot \nabla \lambda_\ast^T}
						{\| \nabla \lambda_\ast \|^2}
		\right]	
		\left[
		I	-	\frac{1}{\gamma} \nabla^2 \lambda_\ast	
		\right] \\
		=
		( V V^T )	
		\left[
		I	-	\frac{1}{\gamma} \nabla^2 \lambda_\ast	
		\right].
\end{equation}
We could deduce a linear convergence result rather quickly, if $\| J(\omega_\ast) \| < 1$. Unfortunately, this is not true in general; indeed the 
right-hand factor in the expression above for $J(\omega_\ast)$ can have its norm larger than one due to the negative eigenvalues
of $\nabla^2 \lambda_\ast$. On the other hand, the projected version of the Jacobian
\[
	V^T J(\omega_\ast) V	=	I	-	\frac{1}{\gamma} V^T \nabla^2 \lambda_\ast V		=	\frac{1}{\gamma} ( \gamma I - {\mathcal H}_V )
\] 
is a contraction, provided $V^T \nabla^2 \lambda_\ast V \succ 0$
(i.e., the sufficient condition for the optimality of $\omega_\ast$, which is slightly stronger than the necessary 
condition \textbf{(3)} in Theorem \ref{thm:KKT}), due to the following lemma.
\begin{lemma}\label{thm:norm_sdp}
Suppose $S \in {\mathbb R}^{k \times k}$ is a symmetric positive semi-definite matrix such that $\| S \| \leq 1$. 
Then $\| I - S \| \leq 1$. Furthermore, if $S$ is positive definite, then $\| I - S \| < 1$.
\end{lemma}
\begin{proof}
Let $v_1, \dots, v_k$ be eigenvectors associated with the eigenvalues $\lambda_1, \dots, \lambda_k$ of $S$ such that
${\mathcal V} = \{ v_1, \dots, v_k \}$ is orthonormal. Note that $\lambda_j \in [0,1]$ for $j = 1,\dots, k$. Expand each 
$v \in {\mathbb R}^k$ of unit norm with respect to ${\mathcal V}$, that is $v =  \alpha_1 v_1 + \alpha_2 v_2 + \dots + \alpha_k v_k$ 
where $\sqrt{\alpha_1^2 + \dots + \alpha_k^2} = 1$. Then
\[
	\| (I - S) v \|		=	\| \alpha_1 (1 - \lambda_1) v_1 + \dots + \alpha_n (1 - \lambda_n) v_n \|
				=	\sqrt{ \alpha_1^2 (1 - \lambda_1)^2  +  \dots  +  \alpha_n^2 (1 - \lambda_n)^2 }
				\leq 1
\]
meaning $\| I - S \| \leq 1$. If $S$ is positive definite, then $\lambda_j \in (0,1]$ for $j = 1,\dots, k$. In this case the argument 
above yields $\| (I - S) v \|	< 1$ and $\| I - S \| < 1$. 
\end{proof}
\begin{theorem}\label{thm:Jacobian_norm}
	Let $\omega_\ast$ be a local maximizer of (\ref{eq:problem}) such that $\lambda_\ast$
	is simple, and $\nabla \lambda_\ast \neq 0$. Then we have $ \frac{1}{\gamma} \| \gamma I - {\mathcal H}_V \| \leq 1$. 
	Additionally, if $V^T \nabla^2 \lambda_\ast V \succ 0$, then $\frac{1}{\gamma} \| \gamma I - {\mathcal H}_V \| < 1$.
\end{theorem}

The projected Jacobian $V^T J(\omega_\ast) V$ assumes the role of $J(\omega_\ast)$ for large $k$, since 
$\omega_k - \omega_\ast$ lies almost on ${\rm Col}(V)$ for such $k$ and has very little component in the direction of 
$u := \nabla \lambda_\ast / \| \nabla \lambda_\ast \|$. This is proven next.
\begin{lemma}\label{lemma:prop_comp}
Suppose that $\{ \omega_{k} \}$ converges to a local maximizer $\omega_\ast$, $\lambda_\ast$ is simple, and
$\nabla \lambda_\ast \neq 0$. Furthermore, suppose 
$\lim_{k\rightarrow \infty} [I - 1/\gamma \nabla^2 \lambda_\ast] [(\omega_k - \omega_\ast)  /\| \omega_k - \omega_\ast \|] \not \in {\rm span}\{ u \}$.
Then $|u^T (\omega_k - \omega_\ast)| / \| V^T (\omega_k - \omega_\ast ) \| \rightarrow 0$ as $k \rightarrow \infty$. 
\end{lemma}
\begin{proof}
Letting $p_k = \omega_{k}  -  \omega_{\ast}$, we have
\[
		p_{k+1}			=	\omega_{k+1}		-	\omega_\ast		=	f(\omega_k)	-	f(\omega_\ast)
						=	
							J_k	\cdot p_k
							\;\;\;\; {\rm where}	\;\;
							J_k 
								=
							\left[
							\begin{array}{c}
								\nabla f_1(\omega_\ast + \eta_1 p_k)^T	\\
								\nabla f_2(\omega_\ast + \eta_2 p_k)^T	\\
								\vdots	\\
								\nabla f_d(\omega_\ast + \eta_d p_k)^T
							\end{array}
							\right]
\]
for some $\eta_1, \dots, \eta_d \in (0,1)$ by the mean value theorem. Above $f_j(\omega)$ denotes the $j$th component of $f(\omega)$.
The recurrence above can be rearranged as
\begin{equation}\label{eq:recurrence}
	\begin{split}
		p_{k+1}	& =	J(\omega_\ast) p_k	+	[J_k - J(\omega_\ast)] p_k		\\
				& =	( V V^T )	
		\left[
		I	-	\frac{1}{\gamma} \nabla^2 \lambda_\ast	
		\right] p_k	+	[J_k - J(\omega_\ast)] p_k	.	
	\end{split}
\end{equation}
By dividing both sides above by $\| p_k \|$, taking the norms and then the limit as $k \rightarrow \infty$ yield
\begin{equation}\label{eq:recurrence2}
	\begin{split}
		\lim_{k\rightarrow \infty}
			\frac{\| p_{k+1} \| }{ \| p_k \| }	 & =	
		\lim_{k\rightarrow \infty}
		\left\|
			( V V^T )	
			\left[
				I	-	\frac{1}{\gamma} \nabla^2 \lambda_\ast	
			\right] \frac{p_k}{ \| p_k \| }	+	[J_k - J(\omega_\ast)] \frac{p_k}{ \| p_k \| }	
		\right\|	\\
								& =
		\lim_{k\rightarrow \infty}
		\left\|
			( V V^T )	
			\left[
				I	-	\frac{1}{\gamma} \nabla^2 \lambda_\ast	
			\right] \frac{p_k}{ \| p_k \| }
		\right\|,
	\end{split}
\end{equation}
where we utilized $\lim_{k\rightarrow \infty} J_k = J(\omega_\ast)$. This means that
$
	\lim_{k\rightarrow \infty}
			\| p_{k+1} \| / \| p_k \|  > 0
$
due to the assumption $\lim_{k\rightarrow \infty} (I - 1/\gamma \nabla^2 \lambda_\ast) (p_k/\| p_k \|) \not \in {\rm span}\{ u \} = {\rm Null}(VV^T)$.

Multiplying both sides of (\ref{eq:recurrence}) by $u^T$ from the left, and taking the absolute value lead us to
\[
	| u^T p_{k+1} |		=		| u^T [J_k - J(\omega_\ast)] p_k |,
\]
thus
\[
	\lim_{k\rightarrow \infty}
	\frac{ | u^T p_{k+1} | }{ \| p_{k} \| }
				=
	\left| u^T [J_k - J(\omega_\ast)] \frac{p_k}{\| p_k \|} \right|
				= 0.
\]
On the other hand 
\[
	0 < \lim_{k\rightarrow \infty} \frac{\| p_{k+1} \|}{\| p_k \|}	=	\lim_{k\rightarrow \infty} \frac{\sqrt{  \| V^Tp_{k+1} \|^2 + | u^T p_{k+1}|^2  }}{\| p_k \|}
												= 	\lim_{k\rightarrow \infty} \frac{ \| V^Tp_{k+1} \| }{\| p_k \|}.
\]
Now the result follows from
\[	
	\lim_{k\rightarrow \infty}   \left(  \frac{ | u^T p_{k+1} | }{ \| p_{k} \| }   \right) / \left( \frac{ \| V^T p_{k+1} \| }{ \| p_{k} \| }  \right)
					=
	\lim_{k\rightarrow \infty}   \frac{ | u^T p_{k+1} | }{ \| V^T p_{k+1} \| }.
\]
\end{proof}

\begin{theorem}[Rate of Convergence]\label{thm:rate_convergence}
Suppose that $\{ \omega_{k} \}$ converges to a local maximizer $\omega_\ast$, $\lambda_\ast$ is simple, and
$\nabla \lambda_\ast \neq 0$. Then one of the following two hold:
 \begin{enumerate}
 \item[\bf (1)] If $\;\lim_{k\rightarrow \infty} [I - 1/\gamma \nabla^2 \lambda_\ast] [(\omega_k - \omega_\ast)  /\| \omega_k - \omega_\ast \|] \in {\rm span}\{ u \}$, then
 \begin{equation}\label{eq:rate_conv1}
 		\lim_{k\rightarrow \infty}
		\frac{\|  \omega_{k+1} - \omega_\ast  \|}{\|  \omega_k - \omega_\ast  \|}	=	0;
 \end{equation}
 \item[\bf (2)] Otherwise,
\begin{equation}\label{eq:rate_conv2}
		\frac{1}{\gamma} \sigma_{\min}( \gamma I - {\mathcal H}_V )
		\leq
		\lim_{k\rightarrow \infty}
		\frac{\|  \omega_{k+1} - \omega_\ast  \|}{\|  \omega_k - \omega_\ast  \|} \leq \frac{1}{\gamma} \| \gamma I - {\mathcal H}_V \|.
\end{equation}
\end{enumerate}
\end{theorem}
\begin{proof}
If $\lim_{k\rightarrow \infty} [I - 1/\gamma \nabla^2 \lambda_\ast] [(\omega_k - \omega_\ast)  /\| \omega_k - \omega_\ast \|] \in {\rm span}\{ u \}$, then
(\ref{eq:rate_conv1}) is an immediate consequence of (\ref{eq:recurrence2}). 

Thus suppose $\lim_{k\rightarrow \infty} [I - 1/\gamma \nabla^2 \lambda_\ast] [(\omega_k - \omega_\ast)  /\| \omega_k - \omega_\ast \|] \not \in {\rm span}\{ u \}$. We exploit 
the recurrence (\ref{eq:recurrence}), and the identities
\begin{equation}\label{eq:limit_identities}
	|u^T (\omega_k - \omega_\ast)| / \| (\omega_k - \omega_\ast ) \| \rightarrow 0,	\;\;	\| V^T (\omega_k - \omega_\ast) \| / \| (\omega_k - \omega_\ast ) \| \rightarrow 1
	\;\;\;\; {\rm as} \;\; k \rightarrow \infty,
\end{equation}
which are immediate corollaries of Lemma \ref{lemma:prop_comp} as 
\[
	\| (\omega_k - \omega_\ast ) \|	
							= \sqrt{ |u^T (\omega_k - \omega_\ast)|^2	+	\| V^T (\omega_k - \omega_\ast) \|^2}.
\]
Dividing the norms of both sides of (\ref{eq:recurrence}) by $\| p_k \|$ we obtain
\[
	\frac{\| p_{k+1 \|}}{\| p_k \|}
			=
	\left\|		
		VV^T
		\left[
		I	-	\frac{1}{\gamma} \nabla^2 \lambda_\ast	
		\right] VV^T \frac{p_k}{\| p_k \|}	+
		VV^T
		\left[
		I	-	\frac{1}{\gamma} \nabla^2 \lambda_\ast	
		\right] uu^T \frac{p_k}{\| p_k \|}		+	
		[J_k - J(\omega_\ast)] \frac{p_k}{\| p_k \|}
	\right\|
\]
where $p_k = \omega_k - \omega_\ast$. Finally, taking the limits of both sides, and utilizing the
identities (\ref{eq:limit_identities}) yield
\[
	\begin{split}
	\lim_{k\rightarrow \infty}	\frac{\| p_{k+1 \|}}{\| p_k \|}
						& =
	\lim_{k\rightarrow \infty}	\left\|		
							VV^T
							\left[
								I	-	\frac{1}{\gamma} \nabla^2 \lambda_\ast	
							\right] VV^T \frac{p_k}{\| p_k \|}
						\right\| \\
						& =
	\lim_{k\rightarrow \infty}	\left\|
							\frac{1}{\gamma} ( \gamma I - {\mathcal H}_V )
							\frac{V^T p_k}{\| p_k \|}
						\right\| \\
						& \leq
	\lim_{k\rightarrow \infty}	\left\|
							\frac{1}{\gamma} ( \gamma I - {\mathcal H}_V )
						\right\|
							\frac{\| V^T p_k \|}{\| p_k \|}
								=	
						\frac{1}{\gamma} 
						\left\|
							  \gamma I - {\mathcal H}_V 
						\right\|.
	\end{split}
\]
The lower bound can be deduced similarly from the inequality
$
						\left\|
							\frac{1}{\gamma} ( \gamma I - {\mathcal H}_V )
							\frac{V^T p_k}{\| p_k \|}
						\right\|
						\geq
						\frac{1}{\gamma} \sigma_{\min}( \gamma I - {\mathcal H}_V ) \frac{\| V^T p_k \|}{\| p_k \|}.
$
\end{proof}

\section{Estimation of an Upper Bound on Second Derivatives}\label{sec:sec_der}
The practicality of the algorithm presented and analyzed relies on the availability of an upper
bound $\gamma$ satisfying (\ref{eq:sd_bound}). The next result is helpful in determining
such a $\gamma$ analytically. An analogous result was proven in \cite[Theorem 6.1]{Mengi2013} for the Hessian
of a weighted sum of the $j$ largest eigenvalues. We include a proof for the sake of completeness. 
\begin{theorem}\label{thm:sec_der_ubound}
Let ${\mathcal A}(\omega) : {\mathbb R}^d \rightarrow {\mathbb C}^{n\times n}$ be a Hermitian
and analytic matrix-valued function. Then
\[
	\lambda_{\max}  \left[ \nabla^2 \lambda_{\min} ({\mathcal A}(\omega))   \right]	
					\leq
		\lambda_{\max}
		\left( 		\nabla^2 {\mathcal A} (\omega)		\right)
\]
for all $\omega \in {\mathbb R}^d$ such that $\lambda_{\min} ({\mathcal A}(\omega))$ is simple, where
\[
	\nabla^2 {\mathcal A}(\omega)
			:=
	\left[
			\begin{array}{cccc}
					\frac{\partial {\mathcal A}^2(\omega)}{\partial \omega_1^2}	&  \frac{\partial {\mathcal A}^2(\omega)}{\partial \omega_1\partial \omega_2}  & \dots  & 	\frac{\partial {\mathcal A}^2(\omega)}{\partial \omega_1 \partial \omega_d}	\\
					\frac{\partial {\mathcal A}^2(\omega)}{\partial \omega_2 \partial \omega_1}  &  \frac{\partial {\mathcal A}^2(\omega)}{\partial \omega_2^2}	& \dots	& \frac{\partial {\mathcal A}^2(\omega)}{\partial \omega_2 \partial \omega_d}	\\
					\vdots	&  \vdots  &	 &  \vdots	\\
					\frac{\partial {\mathcal A}^2(\omega)}{\partial \omega_d \partial \omega_1} & 	\frac{\partial {\mathcal A}^2(\omega)}{\partial \omega_d \partial \omega_2} & \dots	&	\frac{\partial {\mathcal A}^2(\omega)}{\partial \omega_d^2}
			\end{array}
	\right]
\]
\end{theorem}
\begin{proof}
By Theorem \ref{lemma:anal_eig} part \textbf{(iv)} we have 
\[
	\lambda_{\max}  \left[ \nabla^2 \lambda_{\min} ({\mathcal A}(\omega))   \right]	=	
			H_n(\omega) 	+ 	2 \sum_{m=1}^{n-1}  \frac{1}{ \lambda_{\min}({\mathcal A}(\omega)) - \lambda_m(\omega) }  \Re \left( H_{n,m}(\omega) \right)
\]
where the entries of $H_n(\omega)$ and $H_{n,m}(\omega)$ at position $(k,\ell)$ are given by
\[
	v_n^{\ast}(\omega)	\frac{\partial^2 {\mathcal A}(\omega)}{\partial \omega_k \; \partial \omega_l} v_n(\omega)
		\;\; 	{\rm and}	\;\;
	\left(
		v_n(\omega)^{\ast} \frac{\partial {\mathcal A}(\omega)}{\partial \omega_k} v_m(\omega)
	\right) 
	\left(
		v_m(\omega)^{\ast} \frac{\partial {\mathcal A}(\omega)}{\partial \omega_\ell} v_n(\omega)
	\right),
\]
respectively. It is straightforward to verify that $H_{n,m}(\omega)$ is positive semidefinite, since
for each $u \in {\mathbb C}^d$ we have 
\[
	u^\ast H_{n,m}(\omega) u		=		\left| \sum_{\ell = 1}^d h^{(n,m)}_\ell u_\ell  \right|^2	\geq 0
	\;\;\;\;\;\; {\rm where} \;\;
	h^{(n,m)}_\ell	= v_m(\omega)^{\ast} \frac{\partial {\mathcal A}(\omega)}{\partial \omega_\ell} v_n(\omega).
\]
This implies that $\Re \left( H_{n,m}(\omega) \right)$ is also positive semi-definite
due to $u^T \Re \left( H_{n,m}(\omega) \right) u = u^T H_{n,m}(\omega) u \geq 0$ for each $u \in {\mathbb R}^d$.
Thus we deduce
\[
	\lambda_{\max}  \left[ \nabla^2 \lambda_{\min} ({\mathcal A}(\omega))   \right]
				\leq
	\lambda_{\max}(H_n(\omega)) \leq \lambda_{\max}(\nabla^2 {\mathcal A}(\omega)).
\]
Denoting the Kronecker product with $\otimes$, the last inequality above follows from 
$H_n(\omega) = \left[ I_d \otimes v_n^\ast(\omega) \right] \nabla^2 {\mathcal A}(\omega) \left[ I_d \otimes v_n(\omega) \right]$
and the observation that there exists a unit vector $v \in {\mathbb C}^d$ satisfying
\[
	\lambda_{\max}(H_n(\omega)) = v^\ast H_n(\omega) v	=	\left[ v^\ast \otimes v_n^\ast(\omega) \right] \nabla^2 {\mathcal A}(\omega) \left[ v \otimes v_n(\omega) \right]
							\leq \lambda_{\max}(\nabla^2 {\mathcal A}(\omega)).
\]
\end{proof}

\section{Case Study: Pseudospectral Functions}\label{sec:pseudo}
\textbf{Pseudospectral Abscissa:}
The $\epsilon$-pseudospectrum of a matrix $A \in {\mathbb C}^{n\times n}$ is the subset of the complex plane 
consisting of the eigenvalues of all matrices within an $\epsilon$-neighborhood of $A$, formally defined by
\begin{displaymath}
	\Lambda_{\epsilon}(A)	:=	\bigcup_{\| \Delta \| \leq \epsilon} \Lambda(A + \Delta),
\end{displaymath}
with the singular value characterization
\begin{equation}\label{eq:pss_svd}
	\Lambda_{\epsilon}(A) 	=	\{ z \in {\mathbb C} \; | \; \sigma_n(A - zI) \leq \epsilon \},
\end{equation}
when it is defined in terms of the matrix 2-norm \cite{Trefethen2005}. Here $\sigma_n(\cdot)$ denotes
the smallest singular value. The rightmost point in this set
$\alpha_{\epsilon}(A)$ is called the $\epsilon$-pseudospectral abscissa, and is an indicator of
the transient behavior of the dynamical system $x'(t) = Ax(t)$. Globally and locally convergent algorithms for 
$\alpha_{\epsilon}(A)$ have been suggested in \cite{Mengi2005} and \cite{Guglielmi2011}, respectively.
The $\epsilon$-pseudospectral abscissa can be cast as the following optimization problem
\[
	{\rm maximize}_{\omega \in {\mathbb R}^2}	\;\; \omega_1		\;\;\;\;	{\rm subject} \;\; {\rm to} \;\;	
				\lambda_{\min}({\mathcal A}(\omega)) \leq 0
\]
where 
\begin{center}
	${\mathcal A}(\omega) = \left[ A - (\omega_1 + i\omega_2)I \right]^\ast	\left[ A - (\omega_1 + i\omega_2)I \right] - \epsilon^2 I$,
\end{center}
that fits into the framework (\ref{eq:problem}). It follows from the expressions in part \textbf{(iv)} of
Lemma \ref{lemma:anal_eig} that
\[
	\nabla \lambda_{\min}({\mathcal A}(\omega))
				=
	\left( \;
		v_n(\omega)^\ast (-A -A^\ast + 2\omega_1 I) v_n(\omega), \;\;
		v_n(\omega)^\ast (-iA + iA^\ast + 2\omega_2 I) v_n(\omega)
		\;
	\right).
\]
Furthermore, the matrix $\nabla^2 {\mathcal A}(\omega)$ in Theorem \ref{thm:sec_der_ubound} is given by
\[
	\nabla^2 {\mathcal A}(\omega)
				=
	\left[
			\begin{array}{cc}
				\frac{\partial {\mathcal A}^2(\omega)}{\partial \omega_1^2}	&  \frac{\partial {\mathcal A}^2(\omega)}{\partial \omega_1\partial \omega_2} 	\\
				\frac{\partial {\mathcal A}^2(\omega)}{\partial \omega_2 \partial \omega_1}  &  \frac{\partial {\mathcal A}^2(\omega)}{\partial \omega_2^2}
			\end{array}
	\right]
				=
				2 I.
\]
Thus, we deduce $\lambda_{\max}\left[ \nabla^2 \lambda_{\min}({\mathcal A}(\omega)) \right] \leq \gamma := 2$ 
for all $\omega$ such that $\lambda_{\min}({\mathcal A}(\omega))$ is simple by Theorem \ref{thm:sec_der_ubound}.

We run the algorithm on a $10\times 10$ matrix with random entries selected from a normal distribution
with zero mean and variance equal to one to compute $\alpha_{\epsilon}(A)$, as well as to compute the real 
part of the left-most point in $\Lambda_{\epsilon}(A)$ for $\epsilon = 1$. The algorithm is initiated with $\omega_0$ 
equal to the right-most eigenvalue. The progress of the algorithm is illustrated in Figure \ref{fig:psa_psr}
with red and blue asterisks. The algorithm requires 39 iterations to compute both the right-most and the left-most 
points accurate up to 12 decimal digits. In Figure \ref{fig:tangential_progress}, on the left the later iterates of the 
algorithm for the right-most point are shown. Remarkably, at the later iterations the directions $(\omega_k - \omega_\ast)$ 
become more or less orthogonal to $\nabla \lambda_{\min}({\mathcal A}(\omega_\ast))$ (equivalently
tangent to the boundary of $\Lambda_{\epsilon}(A)$), confirming the validity of Lemma \ref{lemma:prop_comp}.

\begin{figure}
		\begin{center}
		\includegraphics[width=.70\textwidth,]{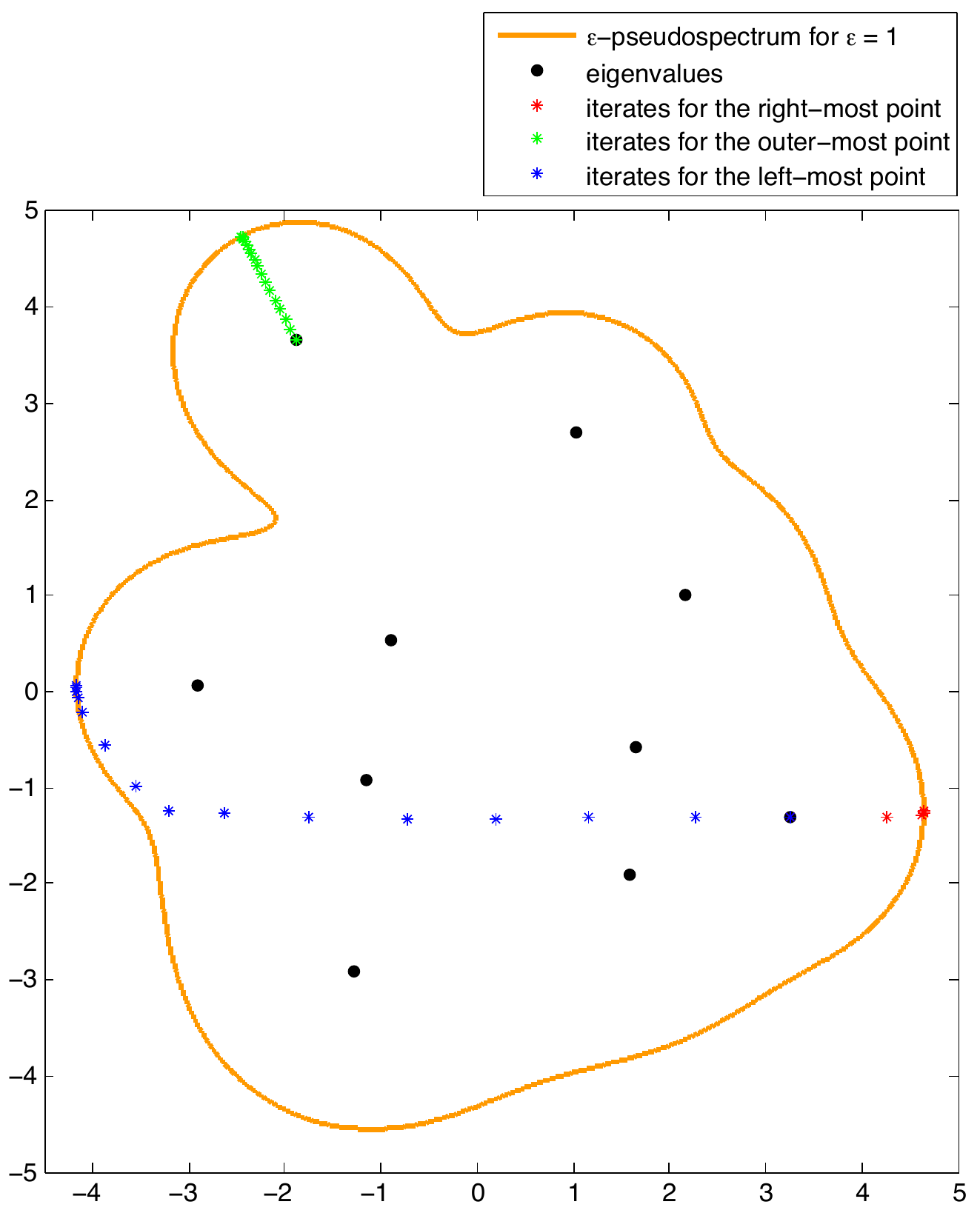} 
		\end{center}
	    \caption{ The progress of the algorithm to compute the right-most, left-most and outer-most points in
	    $\Lambda_{\epsilon}(A)$ is illustrated for a random $10\times 10$ matrix and $\epsilon = 1$. The orange
	    curve represents the boundary of $\Lambda_{\epsilon}(A)$, while the red, blue and green asterisks mark
	    the iterates of the algorithm to compute the right-most, left-most and outer-most points in $\Lambda_{\epsilon}(A)$.
	    The real part of the right-most and the modulus of the outer-most points correspond to $\alpha_{\epsilon}(A)$ 
	    and $\rho_{\epsilon}(A)$, respectively.  }\label{fig:psa_psr}
\end{figure}

\begin{figure}
   	\begin{tabular}{ll}
		\includegraphics[width=.31\textwidth,]{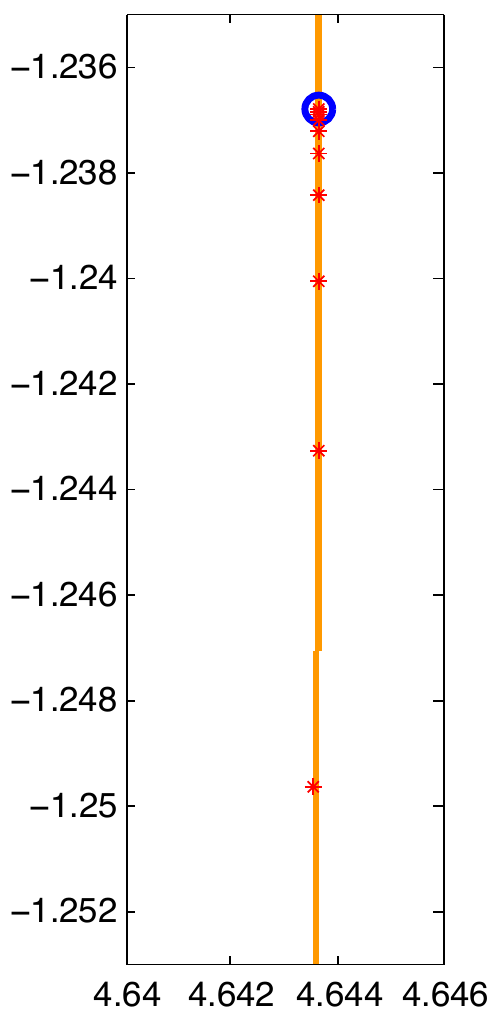} &
		\includegraphics[width=.60\textwidth,]{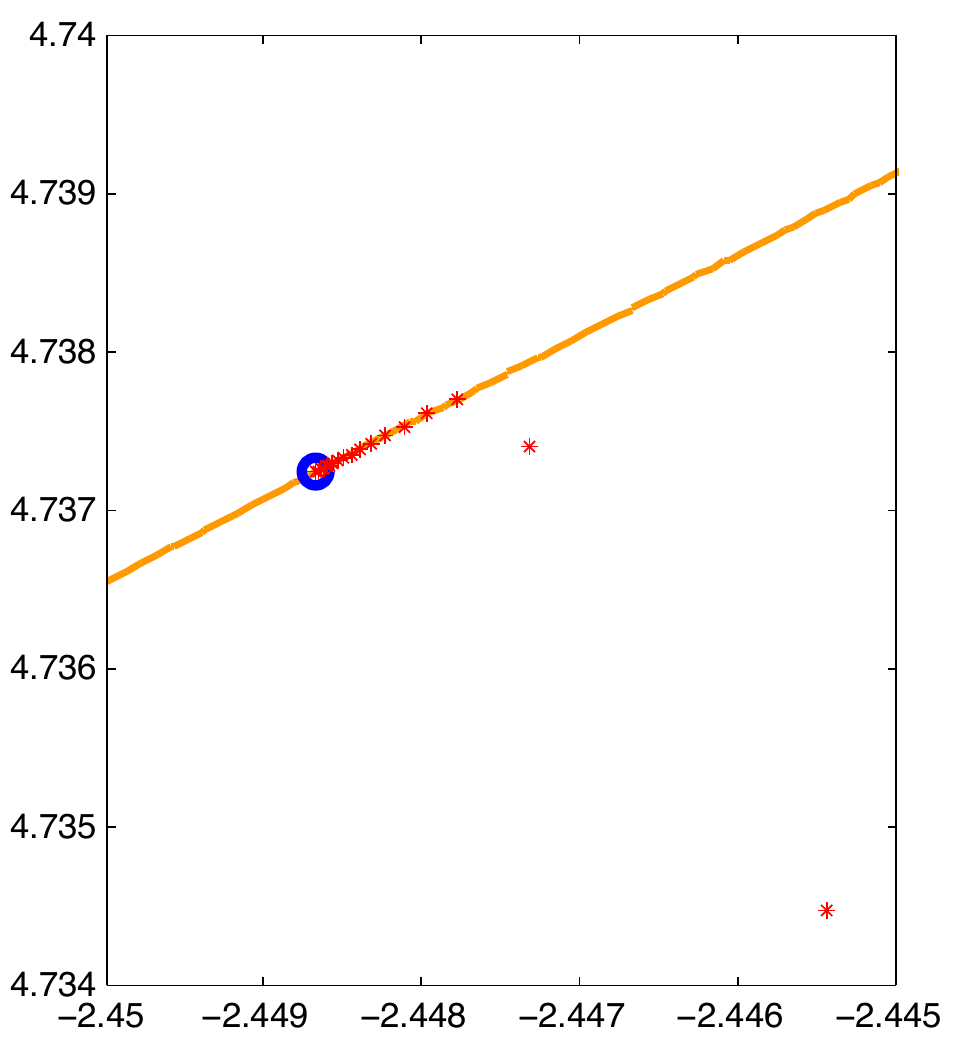} \\
	\end{tabular}
	    \caption{ 
	    Left - A closer look at the iterates for the right-most point; the blue circle marks the right-most point computed 
	    by the highly accurate algorithm in \cite{Burke2003}. Right - Later iterates for the outer-most point; the blue
	    circle marks the outer-most point computed by the algorithm in \cite{Mengi2005}
	    }\label{fig:tangential_progress}
\end{figure}

As for the rate of convergence, the errors $\| \omega_k - \omega_\ast \|$ involved in computing the right-most point 
at later iterations are listed in Table \ref{tab:error}. The table indicates a linear convergence.  Furthermore,
the projected Hessian reduces to the scalar ${\mathcal H}_V =   \frac{ \partial^2 \lambda_{\min}({\mathcal A}(\omega_\ast)) }{\omega_2^2}$, 
thus Theorem \ref{thm:rate_convergence} implies
\[
	\lim_{k\rightarrow \infty}	\frac{\| \omega_{k+1} - \omega_{\ast} \|}{\| \omega_k - \omega_\ast \|}
							=
				\left|	1		-		\frac{1}{\gamma} {\mathcal H}_V 		\right|	=	\left|		1	-	\frac{1}{2} \frac{ \partial^2 \lambda_{\min}({\mathcal A}(\omega_\ast)) }{\partial \omega_2^2}	\right| \approx 0.506.
\]
This ratio of decay in the error is also confirmed in practice by Table \ref{tab:error}.

\begin{table}
\begin{center}
\begin{tabular}{|c||cccccc|}
\hline
$k$							&	28					&	29					&	30					&	31					&	32					& 	33				\\
\hline
$\| \omega_k - \omega_\ast \|$		&	$2.056\cdot 10^{-9}$ 	&	$1.041\cdot 10^{-9}$	&	$5.27\cdot 10^{-10}$	&	$2.67\cdot 10^{-10}$	&	$1.35\cdot 10^{-10}$	&	$6.9\cdot 10^{-11}$	\\
\hline
\end{tabular}
\caption{ The errors $\| \omega_k - \omega_\ast \|$ for various $k$ to compute the right-most point in $\Lambda_{\epsilon}(A)$,
rounded to 12 decimal digits, are listed for the random matrix example and $\epsilon = 1$. } \label{tab:error}
\end{center}
\end{table}

\vskip 2ex

\noindent
\textbf{Pseudospectral Radius:} The modulus of the outermost point $\rho_{\epsilon}(A)$ in the $\epsilon$-pseudospectrum
$\Lambda_{\epsilon}(A)$ is called the $\epsilon$-pseudospectral abscissa \cite{Trefethen2005}. This quantity is associated
with the transient behavior of the discrete dynamical system $x_{k+1} = A x_k$. When $\Lambda_{\epsilon}(A)$ is defined
in terms of the spectral norm and using the singular value characterization (\ref{eq:pss_svd}) we deduce
\[
	\rho_{\epsilon}(A) :=		\;\;\;\;\;
	{\rm maximize}_{\omega \in {\mathbb R}^2}	\;\; \omega_1		\;\;\;\;	{\rm subject} \;\; {\rm to} \;\;	
				\lambda_{\min}({\mathcal A}(\omega)) \leq 0
\]
where 
\begin{center}
	${\mathcal A}(\omega) = \left( A - \omega_1 e^{i\omega_2} I \right)^\ast	\left( A -  \omega_1 e^{i\omega_2} I \right) - \epsilon^2 I$.
\end{center}
Now the expressions for the first derivatives take the form
\[
	\nabla \lambda_{\min}({\mathcal A}(\omega))
				=
	\left( \;
		v_n(\omega)^\ast (-2\Re(e^{-i\omega_2} A) + 2\omega_1 I) v_n(\omega), \;\;
		v_n(\omega)^\ast (-2\Im(\omega_1 e^{-i \omega_2} A) ) v_n(\omega)
		\;
	\right),
\]
whereas
\[
	\nabla^2 {\mathcal A}(\omega)
				=
	\left[
			\begin{array}{cc}
				2I					&	-2\Im(e^{-i\omega_2}A)	\\
				-2\Im(e^{-i\omega_2}A)	&	2\Re(\omega_1 e^{-i\omega_2} A) \\
			\end{array}
	\right].
\]
Since $\rho_{\epsilon}(A) \leq \| A \| + \epsilon$, for all feasible $\omega$ we have 
$\omega_1 \leq \| A \| + \epsilon$. Thus, Gersgorin's theorem \cite[Theorem 6.1.1]{Horn1990} applied to
$\nabla^2 {\mathcal A}(\omega)$, combined with Theorem \ref{thm:sec_der_ubound} 
yields 
\[
	\lambda_{\max}  \left[	\nabla^2 \lambda_{\min} ({\mathcal A}(\omega))	  \right]
				\leq
	\gamma := \max\left( 2 + 2\| A \|,  2\epsilon \| A \| + 2 \| A \|^2 + 2 \| A \| \right)
\]
for all feasible $\omega$ such that $\lambda_{\min} \left( {\mathcal A}(\omega) \right)$ is simple.

We apply the algorithm to compute $\rho_{\epsilon}(A)$ starting with $\omega_0$ equal to the
eigenvalue with the largest modulus. The iterates of the algorithm are shown in Figure \ref{fig:psa_psr}
with the green asterisks for the $10\times 10$ random matrix of the previous part concerning 
the computation of $\alpha_{\epsilon}(A)$, and $\epsilon = 1$. The later iterations,
as illustrated on the right in Figure \ref{fig:tangential_progress}, again become
tangential to the boundary of $\Lambda_{\epsilon}(A)$, which is in harmony with Lemma \ref{lemma:prop_comp}.

For the rate of convergence, by Theorem \ref{thm:rate_convergence}, we obtain
\[
	\lim_{k\rightarrow \infty}	\frac{\| \omega_{k+1} - \omega_{\ast} \|}{\| \omega_k - \omega_\ast \|}
							=
				\left|	1		-		\frac{1}{\gamma} {\mathcal H}_V 		\right|	=	\left|		1	-	\frac{1}{\gamma} \frac{ \partial^2 \lambda_{\min}({\mathcal A}(\omega_\ast)) }{\partial \omega_2^2}	\right| \approx 0.791.
\]
This is confirmed by Table \ref{tab:error2} below, which indicates a linear convergence with the ratio of two
consecutive errors roughly equal to 0.791.

\begin{table}
\begin{center}
\begin{tabular}{|c||cccccc|}
\hline
$k$							&	41					&	42					&	43					&	44					&	45					& 	46				\\
\hline
$\| \omega_k - \omega_\ast \|$		&	$1.105\times 10^{-7}$ 	&	$8.742\cdot 10^{-8}$	&	$6.918\cdot 10^{-8}$	&	$5.474\cdot 10^{-8}$	&	$4.332\cdot 10^{-8}$	&	$3.428\cdot 10^{-8}$	\\
\hline
\end{tabular}
\caption{ The errors $\| \omega_k - \omega_\ast \|$ for various $k$ to compute the outer-most point in $\Lambda_{\epsilon}(A)$
are listed for the random matrix example and $\epsilon = 1$. } \label{tab:error2}
\end{center}
\end{table}

\section{Conclusion}
There are quite a few applications that require optimization of linear functions subject to
a minimum eigenvalue constraint, or optimization problems of similar spirit, such as the 
calculation of the pseudospectral functions, shape optimization problems in structural
design, and robust stability problems in control theory. We explored the support based 
global optimization ideas for such problems. The use of quadratic support functions
benefitting from derivatives yield a simple linearly convergent algorithm robust against
the non-smooth nature of the eigenvalue functions. We establish the convergence
assuming the simplicity of the eigenvalue at the optimal point, which we believe is
not essential. In any case, the algorithm is immune to non-smoothness at points close
to optimal points. The rate of the convergence of the algorithm is analyzed in detail
leading us to a fine understanding of the factor affecting the rate of convergence, 
basically the eigenvalue distribution of a projected Hessian matrix.

The algorithm suggested may be applicable in other occasions when the constraint
involves an eigenvalue function other than the smallest one. We should, however, note
practical difficulties in this more general setting, for instance a global lower bound on the 
second derivatives may not be easy to deduce.

The algorithm suggested for constrained eigenvalue optimization is locally convergent, 
while the counter-part for unconstrained optimization in an earlier 
work \cite{Mengi2013} is globally convergent. It is worth pursuing how the algorithm 
can be modified in order to achieve global convergence.

\vskip 2ex

\noindent
\textbf{Acknowledgements} \\
The author is grateful to Emre Alper Y{\i}ld{\i}r{\i}m for fruitful discussions.

\bibliography{constrained_eigopt}

\begin{thebibliography}{10}

\bibitem{Achtziger2007}
W.~Achtziger and M.~Kocvara.
\newblock Structural topology optimization with eigenvalues.
\newblock {\em SIAM J. Optim.}, 18(4):1129--1164, 2007.

\bibitem{Burke2003}
J.V. Burke, A.S. Lewis, and M.L. Overton.
\newblock Robust stability and a criss-cross algorithm for pseudospectra.
\newblock {\em IMA J. Numer. Anal.}, 23(3):359--375, 2003.

\bibitem{Guglielmi2011}
N.~Guglielmi and M.L. Overton.
\newblock Fast algorithms for the approximation of the pseudospectral abscissa
  and pseudospectral radius of a matrix.
\newblock {\em SIAM J. Matrix Anal. Appl.}, 32(4):1166--1192, 2011.

\bibitem{Horn1990}
R.A. Horn and C.A. Johnson.
\newblock {\em Matrix Analysis}.
\newblock Cambridge University Press, 1990.

\bibitem{Kelley1960}
J.E. Kelley.
\newblock The cutting-plane method for solving convex programs.
\newblock {\em J. Soc. Ind. Appl. Math.}, 8(4):pp. 703--712, 1960.

\bibitem{Kiwiel1985}
K.C. Kiwiel.
\newblock {\em Methods of descent for nondifferentiable optimization}.
\newblock Lecture notes in mathematics. Springer-Verlag, Berlin, New York,
  1985.

\bibitem{Lancaster1964}
P.~Lancaster.
\newblock On eigenvalues of matrices dependent on a parameter.
\newblock {\em Numer. Math.}, 6:377--387, 1964.

\bibitem{Makela2002}
M.~M\"akel\"a.
\newblock Survey of bundle methods for nonsmooth optimization.
\newblock {\em Optim. Method. Softw.}, 17(1):1--29, 2002.

\bibitem{Mengi2005}
E.~Mengi and M.L. Overton.
\newblock Algorithms for the computation of the pseudospectral radius and the
  numerical radius of a matrix.
\newblock {\em IMA J. Numer. Anal.}, 25(4):648--669, 2005.

\bibitem{Mengi2013}
E.~Mengi, E.A. Yildirim, and M.~Kilic.
\newblock Numerical optimization of eigenvalues of {H}ermitian matrix
  functions.
\newblock {\em SIAM J. Matrix Anal. Appl.}, 2013.
\newblock Submitted.

\bibitem{Rellich1969}
F.~Rellich.
\newblock {\em Perturbation Theory of Eigenvalue Problems}.
\newblock Gordon and Breach, 1969.

\bibitem{Trefethen2005}
L.~N. Trefethen and M.~Embree.
\newblock {\em Spectra and Pseudospectra. The Behavior of Nonnormal Matrices
  and Operators}.
\newblock Princeton University Press, Princeton, NJ, 2005.

\end{thebibliography}
\end{document}